\documentclass[11pt]{article}
\usepackage{amsmath} 
\usepackage{amsthm} 
\usepackage{amssymb}
\usepackage[T1]{fontenc} 
\usepackage{comment} 
\usepackage[english]{babel} 
\theoremstyle{plain}
\newtheorem{thm}{Theorem}

\newtheorem*{thm*}{Theorem}
\newtheorem{Lem}{Lemma}
\newtheorem*{Lem*}{Lemma}

\theoremstyle{remark}

\newtheorem*{rmks*}{Remarks}

\numberwithin{equation}{section}
\usepackage{graphicx}

\usepackage{hyperref}

\title{Non-homothetic convex ancient solutions for flows
	by high powers of curvature}
\date{}
\begin{document}
	 	\author{Susanna Risa  \footnote{Dipartimento di Matematica e Applicazioni ``Renato Caccioppoli'', Universit\`a degli Studi di Napoli ``Federico II'', Via Cintia, Monte S.Angelo I-80126, Napoli, Italy. E-mail: susanna.risa@unina.it}
		, Carlo Sinestrari \footnote{Dipartimento di Matematica, Universit\`a di Roma ``Tor Vergata'', Via della Ricerca Scientifica 1, 00133, Roma, Italy. E-mail: sinestra@mat.uniroma2.it}
	}
	\maketitle
	\begin{abstract}	 
	We prove the existence of closed convex ancient solutions to curvature flows which become more and more oval for large negative times. The speed function is a general symmetric function of the principal curvatures, homogeneous of degree greater than one. This generalises previous work on the mean curvature flow and other one-homogeneous curvature flows.	
As an auxiliary result, we prove a new theorem on the convergence to a round point of convex rotationally symmetric hypersurfaces satisfying a suitable constraint on the curvatures.
		\end{abstract}
	\section{Introduction}
	We study ancient solutions for flows of hypersurfaces in Euclidean space, driven by a symmetric function of the principal curvatures. A solution is called ancient if it exists for all negative times in the past; the simplest example is given by a standard sphere shrinking by homotheties. In this article, we study the existence of compact convex ancient solutions which are different from the sphere. 
We construct rotationally symmetric solutions which shrink to a round point at the final time and become more and more eccentric as  $t \to -\infty$. Our result applies to a large class of nonlinear flows: the speed is any symmetric, homogeneous function of the principal curvatures, with degree of homogeneity $\alpha \geq 1$, with no further structural assumptions except the standard monotonicity property which ensures parabolicity. This generalises previous work on the mean curvature flow \cite{W03,WA,HH} and for flows with one-homogeneous speeds satisfying suitable convexity or concavity assumptions \cite{R,LZ}. As a preliminary step in our construction, we prove a result of independent interest on the convergence to a round point of convex rotationally symmetric hypersurfaces such that the axial curvature does not exceed the radial ones.

	Ancient solutions arise as tangent flows of general solutions at points where the curvature becomes unbounded. For this reason, they occur naturally in the study of singularities and their analysis is an essential step in the papers by Hamilton and Perelman on the Ricci flow, for example \cite{H2} and \cite{P} . For extrinsic curvature flows of the kind considered here, ancient solutions have also been studied during the years by many authors, who have constructed examples with various different behaviours and obtained classification results. In addition to the applications to the singularity analysis, ancient solutions have attracted interest in themselves because of their remarkable, and sometimes unexpected, geometric properties, see e.g. \cite{BLT4}. 

The extrinsic flow with the richest collection of results on ancient solutions is the mean curvature flow, particularly in the context of mean convex hypersurfaces. We describe briefly some of the available existence results; we omit the special case of the curve shortening flow in the plane and focus on the higher dimensional setting which is of interest for our paper. Natural examples of ancient solutions are provided by solitons, the solutions which only evolve by a symmetry of the system. The shrinking sphere is the only mean-convex and compact soliton; if one of these two assumptions is removed, further examples arise, such as Angenent's self-shrinking torus \cite{An1}, which is compact, but not mean-convex, or translating, convex and noncompact solutions, like the grim hyperplanes, the bowl and the flying wings \cite{AW,WA,HIMW}.

When passing from solitons to general ancient solutions, one again observes strong rigidity properties in the convex compact case. It has been shown that there are several possible additional conditions, such as uniform curvature pinching, or a growth bound  on the curvature or diameter as $t \to -\infty$, under which no compact convex ancient solution exists other than the shrinking sphere, see e.g. \cite{HH,HS,L1}. On the other hand, if no such additional requirement is imposed, many interesting nontrivial examples have been found. By the results of Wang \cite{WA}, convex ancient solutions of the mean curvature flow either sweep the whole space, or are confined in the slab between two parallel hyperplanes. Following \cite{BLT3}, we call the former ones (except for the round sphere) {\em ovaloids}, and the latter ones {\em pancakes}. The existence of ovaloids was first shown by White \cite{W03}, who sketched an approximation procedure by a sequence of convex solutions of increasing eccentricity. A detailed argument was later provided  by Haslhofer and Hershkovits in \cite{HH}. Ancient pancakes, on the other hand, were constructed by Bourni, Langford and Tinaglia in \cite{BLT} as limits of rotated grim reapers. Both existence results can also be recovered from the more elaborate analysis of Wang in \cite{WA}. As time tends to $-\infty$, suitable blowdowns of the ovaloids converge to the cylinder, while the pancakes converge to the boundary of the slab they live in. By taking limits at the tips, both asymptotically resemble translators: bowl solitons for the ovaloids, grim hyperplanes for the pancake, see \cite{BLT3} for details.

By contrast, much less is known about ancient solutions to hypersurface flows driven by other curvature functions. Results on the existence of ovaloids have been obtained independently in the first author's PhD thesis \cite{R} and, in a more general setting, by Lu and Zhou \cite{LZ}. In both cases, the technique of \cite{W03,HH} is adapted to the case of a speed which has homogeneity one and satisfies suitable assumptions of convexity, concavity, or behaviour on the boundary of the positive cone. To our knowledge, the only related result for other degrees of homogeneity is the one in \cite{B+}, where the authors study the flow of curves in the plane by powers of the curvature, and prove the existence and uniqueness of a nonround compact convex ancient solution for a suitable range of exponents less than one. On the other hand, there are also rigidity results for the spherical solution under hypotheses similar to the ones of the mean curvature flow, see \cite{RS,LL}.

In this paper, we construct ancient ovaloids for a much larger class of flows, by allowing a general homogeneity greater than one and removing all assumptions on the speed except for positivity and parabolicity; on the other hand, we restrict ourselves to the rotationally symmetric case, instead of the more general symmetries considered in \cite{HH,LZ}.
As in the previous works, the ancient flow is obtained as the limit of a sequence of approximate solutions on finite time intervals of increasing length, with initial data given by a long cylinder, smoothly capped at the ends with two half-spheres. We are able to prove that the sequence satisfies uniform bounds on the radii and on the curvature on any finite time interval, which imply compactness and the existence of an ancient limit with the desired properties. It is interesting to observe that the procedure can be carried out even though some important tools from the previous works are no longer available, such as the noncollapsing property for the mean curvature flow and its consequences used in \cite{HH}, or the convexity/concavity assumptions made in \cite{LZ,R} in order to apply Krylov-Safonov's estimates. In fact, it turns out that the special structure of the approximants allows to derive curvature bounds and regularity from simple direct arguments which apply to general speeds, taken from \cite{AMCZ}, with no need of advanced results from parabolic theory. This suggests that the existence of ovaloids is a quite general feature of curvature flows and is not related to particular properties of the speed, except possibly for some growth constraint. On the other hand, it is interesting to remark that there are geometric flows in different settings where no analogue of the ovaloids exists: this has been shown for flows in the sphere in \cite{BIS} and for expanding flows in Euclidean space in \cite{RS2}.

Compared to the previous works, our analysis requires an additional step, which we now describe. The procedure of \cite{W03} makes an essential use of the property that the approximants converge to a round point, that is, they shrink to a point and converge to a round sphere after rescaling. For the mean curvature flow and for the one-homogeneous flows considered by the previous authors, convergence to a round point holds for any convex hypersurface, according to well-known result by Huisken \cite{Hu} and Andrews \cite{A1}. By contrast, in homogeneity greater than one and in general dimension, a result of this kind is not known except for the special case of the powers of the Gauss curvature \cite{AGN, BrendleDaskaChoi}. Furthermore, in \cite{AMCZ} the authors have constructed an example of a convex hypersuface which loses convexity along the evolution if the speed does not satisfy a concavity condition on the boundary of the positive cone. Convergence to a round point holds for general speeds \cite{AMC} if the starting hypersurface satisfies a suitably strong curvature pinching: however, ovaloids do not satisfy any uniform pinching as $t \to -\infty$, thus no such property can be expected on the approximants. Even in the rotationally symmetric case, the few available results in homogeneity greater than one \cite{MMM,LWW} require some curvature pinching. However, we can exploit a further feature of our approximants, namely that they can be constructed in such a way that the axial curvature is not larger than the other ones at every point. It turns out that this property is invariant under the flow and provides the control on the gradient terms in the curvature evolution equations required to prove convergence to a round point by an adaptation of the method of \cite{AMC}. We call a rotationally symmetric hypersurface with this property {\em axially stretched}. We state and prove the convergence of such manifolds to a round point as a separate result, independent from the application to the existence of ovaloids. We point out that the
result holds under such general hypotheses on the speed due to the restrictions on the class of data considered. In fact, the 
counterexample to the invariance of convexity in \cite{AMCZ} is not rotationally symmetric. In addition, Andrews has proved in \cite{A5} that there are rotationally symmetric surfaces for which curvature pinching fails to improve under flows with high homogeneity; such surfaces, however, are not axially stretched.

Our paper is organised as follows. In Section 2 we state the precise assumptions on our speed function and recall some preliminary results on our flow and on rotationally symmetric hypersurfaces. In Section 3 we prove the result on the convergence to a round point of a capsule/axially stretched hypersurface. Finally, in Section 4, we introduce our approximants and prove the uniform second order bounds which ensure compactness and convergence of a subsequence to an ancient ovaloid, which is also shown to be asymptotically cylindrical after performing an appropriate blowdown limit.

	  \section{Preliminaries and notation}
	The flow we consider is expressed as
	\begin{equation}\label{Fflow}
		\frac{\partial \varphi}{\partial t}(p,t)=-f(p,t)\nu(p,t),
			\end{equation}                  
	where $\varphi: M^n \times I \to \mathbb{R}^{n+1}$ is a time-dependent immersion of a hypersurface in Euclidean space, with $n \geq 2$, and $f$ is a function of the principal curvatures at $(p,t)$. The interval of times $I$ will be either bounded or unbounded.

We denote the metric associated with the immersion by $g=\left\{g_{ij}\right\}$, the second fundamental form by $h=\left\{h_{ij}\right\}$, and the principal curvatures by $\lambda_1,\dots,\lambda_n$. Then the mean curvature is given by $H=\lambda_1+\dots+\lambda_n$ and the squared norm of $h$ is $|h|^2=\lambda_1^2+\dots+\lambda_n^2$.

Let us denote by $\Gamma_0 \subset \mathbb{R}^n$ the cone which consists of all $n$-ples of the form $(\lambda,\mu,\mu,\dots,\mu)$ with $\mu>0$ and $0 \leq \lambda \leq \mu$, and of their permutations. This cone describes the possible values of the curvatures for the hypersurfaces we will consider in this paper, which will be convex (possibly weakly), rotationally symmetric and with the curvature in the axial direction not greater than the ones in the radial direction.

We assume that the speed $f=f(\lambda_1,\dots,\lambda_n)$ satisfies the following assumptions.
\begin{description}
\item[(F1)] $f$ is a smooth symmetric function defined on an open symmetric cone $\Gamma$ which contains $\Gamma_0$.
\item[(F2)] $f$ satisfies $\frac{\partial f}{\partial \lambda_i} > 0$ for every $i$ on $\Gamma$.
\item[(F3)] $f$ is homogeneous of degree $\alpha\geq1$.
\item[(F4)] $f$ is positive on $\Gamma$.
\end{description}
Without loss of generality, we will also assume that the speed satisfies the normalisation condition $f(1,\dots,1)=n^\alpha$.
We remark that we do not assume any property of convexity or concavity for $f$.

	 With the capital letter $F$ we denote the speed expressed as a function of the Weingarten operator $h^i_j=h^{ik}g_{kj}$ or, equivalently, of the metric and the second fundamental form; we recall that $f$ and $F$ have the same differentiability properties, see for example \cite{Ger}. As in \cite{AMC,Ger} we use the notation $\dot F^{ij}, \ddot F^{ij,rs}$ for the first and second derivatives of $F$ with respect to $h_{ij}$. We recall that $\dot F^{ij}$ is positive definite by assumption (F2).

	We denote by $\mathcal{L}=\dot{F}^{ij}\nabla_i \nabla_j$ the elliptic operator associated to \eqref{Fflow}.

	The evolution equations for the Weingarten operator $h^i_j$ and for a smooth symmetric function $G$ of the principal curvatures are the following (see, for example, \cite{A1}, \cite{AMC}):
	\begin{align}
		\frac{\partial h^i_j}{\partial t}=& \mathcal{L}h^i_j+\ddot{F}^{kl,rs}\nabla^i h_{kl}\nabla_j h_{rs}+\dot{F}^{kl}h_{km}h^m_lh^i_j+(1-\alpha)Fh^{im}h_{mj},\\
		\frac{\partial G}{\partial t}= &\mathcal{L}G +\left[\dot{G}^{ij}\ddot{F}^{kl,rs}-\dot{F}^{ij}\ddot{G}^{kl,rs}\right]\nabla_i h_{kl}\nabla_jh_{rs} \nonumber \\
		&+\dot{F}^{kl}h_{km}h^m_l\dot{G}^{ij}h_{ij}+(1-\alpha)F\dot{G}^{ij}h_{im}h^m_j.\label{omogeq}
	\end{align}
In particular, the mean curvature and the speed satisfy 
	\begin{align}
		\frac{\partial H}{\partial t}&=\mathcal{L}H+\ddot{F}^{ml,rs}\nabla^i h_{ml}\nabla_i h_{rs}+\dot{F}^{ml}h_{ms}h^s_lH+(1-\alpha)F|h|^2\\
		\frac{\partial F}{\partial t}&=\mathcal{L}F+\dot{F}^{ij}h_{im}h^{m}_jF. \label{speedevo}
	\end{align}

	A function that we will often examine is $Z_\sigma=|h|^2 -\left(\frac 1n+\sigma \right) H^2$, for suitable choices of $\sigma \in \left[0,\frac {1}{n(n-1)}\right]$. In this case, the reaction terms which appear in \eqref{omogeq} take the form
	             \begin{eqnarray}            
	             \lefteqn{\dot{F}^{kl}h_{km}h^m_l\dot{Z_\sigma}^{ij}h_{ij}+(1-\alpha)F\dot{Z_\sigma}^{ij}h_{im}h^m_j} \nonumber \\
	             & = & 2 \dot{F}^{ij}h_{im}h^m_jZ_\sigma+\frac 2n (1-\alpha) F(nC-(1+n\sigma)H|h|^2), \label{reactZ}
	                         \end{eqnarray}   
where $C=\lambda_1^3+\dots+\lambda_n^3$. To estimate the quantity above it is useful to recall Lemma 2.2 from \cite{AMC}, which states the following: at any 
point where $Z_\sigma=0$ we have
\begin{equation}\label{est-reaction}
nC-(1+n \sigma)H|h|^2 \geq \sigma(1+n\sigma)(1-\sqrt{n(n-1)\sigma})H^3.
\end{equation}

To parametrise a convex and rotationally symmetric hypersurface, we can choose $\psi: I \times\mathbb{S}^{n-1}\to \mathbb{R}^{n+1}$ of the form $\psi(x,\omega)=(x,u(x)\omega)$, where $I \subset  \mathbb{R}$ is a bounded open interval, $\omega$ is a coordinate on $S^{n-1}$ and $u$ is a positive real function on $I$ tending to zero at the endpoints. 
The poles are not covered by this map: for our purposes, it is enough to observe that they are umbilical points, by symmetry. In the above parametrisation,
	the principal curvatures are given by
	\begin{equation}\label{curvatures}
	\lambda_1=\frac{-u_{xx}}{(1+u_x^2)^{3/2}}, \quad\quad \lambda_i=
	\frac{1}{u \sqrt{1+u_x^2}}, \qquad i=2,\dots,n.
	\end{equation}
For simplicity of notation, we will set $\lambda:=\lambda_1$ and $\mu:=\lambda_2=\dots=\lambda_n$. We call $\lambda$ and $\mu$ the {\em axial} and {\em radial} curvature respectively.

The following result gives a useful expression for the gradient terms which appear in equation \eqref{omogeq} at a stationary point of $G$ on a rotationally symmetric surface. In the statement below, the apex $i$ will denote derivation with respect to the $i$-th principal curvature, where $i=1$ corresponds to the axial direction and $i=2,\dots,n$ to the radial ones.

 \begin{Lem}\label{LemmaLWW}
			Assume $F,G$ are two smooth homogeneous symmetric functions of the principal curvatures of a convex, closed and axially symmetric hypersurface in $\mathbb{R}^{n+1}$; let $F$ be homogeneous of degree $a$ and $G$ homogeneous of degree $b$.
			At any stationary point for $G$ which is non-umbilical and such that $\dot{g}^{1} \neq 0$, there holds:
			\begin{align}
				(&\dot{G}^{ij}\ddot{F}^{lm,rs}-\dot{F}^{ij}\ddot{G}^{ml,rs})\nabla_i h_{ml}\nabla_j h_{rs}=\nonumber \\
				&\left[\dot{g}^{1}\frac{a(a-1)F}{\mu^2}-\dot{f}^{1} \frac{b(b-1)G}{\mu^2}-2 \dot{g}^{1}\frac{aF}{\mu(\lambda - \mu)}+2\dot{f}^1 \frac{bG}{\mu(\lambda - \mu)} \right. \nonumber\\
				&+ \left(\frac{b^2G^2}{\mu^2(\dot{g}^{1})^2}-\frac{2bG\lambda}{\mu^2\dot{g}^{1}}\right)(\dot{g}^{1}\ddot{f}^{11}-\dot{f}^{1}\ddot{g}^{11})
				\nonumber\\
				&
				\left. -2(n-1)\frac{bG}{\mu \dot{g}^{1}}(\dot{g}^{1}\ddot{f}^{12}-\dot{f}^{1}\ddot{g}^{12})\right](\nabla_1 h_{22})^2 \label{evofg}
			\end{align}
		\end{Lem}

The above identity is proved in \cite[Lemma 3.6]{LWW}; previous similar results can be found in \cite{A5} and \cite{MMM}. 

We remark that the statement in \cite{LWW} requires the hypersurface to have strictly positive curvatures. However, it is easy to see that the result also holds in the weakly convex case. In fact, the proof in \cite{LWW} only requires $\mu>0$ and also works if $\lambda=0$. On the other hand, by \eqref{curvatures},  $\mu>0$ everywhere on a convex axially symmetric hypersurface except possibly at the poles. The poles, however, cannot occur in the above statement since they are umbilical points.

\section{Evolution of axially stretched hypersurfaces}

We will focus our analysis on those convex and rotationally symmetric hypersurfaces with the additional property that the axial curvature $\lambda$ is not greater than the radial one $\mu$. For simplicity, we give a name to this class of manifolds. We say that a compact convex rotationally symmetric hypersurface is {\em axially stretched} if it satisfies $0 \leq \lambda \leq \mu$ and $\mu >0$ everywhere. Equivalently, we can say that  the curvatures belong to the cone $\Gamma_0$ defined in the previous section.

In this section we prove the following result.

	\begin{thm}\label{roundpoint}
	Let $M_0$ be a closed convex (possibly weakly) hypersurface. Suppose that $M_0$ is rotationally symmetric and axially stretched. Then $M_t$ remains axially stretched throughout the evolution by the flow \eqref{Fflow} and converges to a round point in finite time. 
	\end{thm}

Throughout the section, we will assume $\alpha > 1$; the result in the case $\alpha = 1$ has been proved for a general axially symmetric hypersurface in \cite{MMM}. 

 By assumption, the curvatures of $M_0$ lie in a compact subset of the cone $\Gamma$ where the speed $F$ is defined.
Therefore, short time existence of a solution of equation \eqref{Fflow} is granted by standard parabolic theory. On the other hand, a general result on the long time existence is not available, unless some additional convexity or concavity condition on the speed is required. We cannot use the argument of \cite{AMC} either, since it relies on the strong curvature pinching which is assumed in that paper. However, in our case this issue is greatly simplified by the assumption of rotational symmetry. In this setting, in fact, there are derivative estimates, see e.g. Section 4 in \cite{McMMV} and the references therein, yielding regularity of the flow as long as the curvature is bounded. It follows that the solution of \eqref{Fflow} exists up to a finite maximal time at which either the norm $|h|^2$ blows up, or else the curvatures reach the boundary of the cone $\Gamma$.

As a first step, we show that the property of being axially stretched is preserved along the evolution.
\begin{Lem}\label{poles}
Let $M_t$ be a closed, convex and rotationally symmetric solution of \eqref{Fflow} for $t \in [0,T)$. If $M_0$ is axially stretched, then it remains so for all $t \in [0,T)$. In addition, the mean curvature of $M_t$ is bounded from below by a positive constant.
\end{Lem}
\begin{proof} We first prove the claim on the lower bound for $H$. We observe that, by equation \eqref{speedevo}, the minimum of $F$ on $M_t$ is nondecreasing in time, and therefore bounded from below by a positive constant. Since $F$ and $H^\alpha$ are both positive functions on the cone $\Gamma_0$ and have the same degree of homogeneity, a standard compactness argument shows that there are constants $m_1,m_2>0$ such that $m_1 H^\alpha \leq F \leq m_2 H^\alpha$ everywhere in $\Gamma_0$. It follows that $H$ is also bounded from below by a positive constant on $M_t$ during the flow, as claimed.

We now show that the property $\lambda \leq \mu$ is preserved during the flow. Let $A(t)=\left\{p\in M_t \,|\, (\lambda-\mu)(p) > 0\right\}$; by assumption, $A(0)=\emptyset$.  We argue by contradiction and suppose that $A(t)$ is nonempty for some positive time. We let 
$t_0=\inf \left\{t \in (0,T)\,|\,A(t)\neq \emptyset \right\}$.  

Observe that $|h|^2-\frac 1n H^2$ is positive on $A(t)$ and zero on its boundary. Hence, if we choose $\sigma>0$ small enough and we consider the function $Z_\sigma=|h|^2-(\frac 1n+\sigma) H^2$, we can find a first time $\tau > t_0$ at which
$Z_\sigma$ restricted on $A(t)$ attains a zero maximum at some point $p$.  We observe that, by continuity, $(q,t)\in A(t)$ for  $(q,t)$ in a space-time neighbourhood of $(p,\tau)$.  We then study the sign of the terms at the right-hand side of the evolution equation  \eqref{omogeq} with $G=Z_\sigma$. Since $\sigma$ is small, we have a strong curvature pinching and we can follow the proof of Theorem 5.1 in \cite{AMC}. It is proved there that the gradient terms at the point $(q,\tau)$ admit an estimate of the form
\begin{eqnarray*}
	&&\left[\dot{Z_\sigma}^{ij}\ddot{F}^{kl,rs}-\dot{F}^{ij}\ddot{Z_\sigma}^{kl,rs}\right]\nabla_i h_{kl}\nabla_jh_{rs} \\
& \leq & 
2\left(\mu \sqrt{\sigma+(1+n \sigma)}-c_1(\alpha-\mu \sqrt{\sigma})+c_2 \sigma(\alpha+\mu \sqrt{\sigma}) \right)  H^{\alpha-1}|\nabla h|^2,
\end{eqnarray*}
for suitable constants $c_1,c_2>0$ depending on $n$ and $\mu>0$ depending on $F$. 
By choosing $\sigma>0$ suitably small, all terms inside the parentheses become small except for the negative one $-c_1 \alpha$, and thus the total contribution is negative. 

The reaction terms also give a non positive contribution, as it follows from
	\eqref{omogeq}, \eqref{reactZ} \eqref{est-reaction}, taking into account that $\alpha>1$, $F>0$ and $\sigma<\frac{1}{n(n-1)}$. We thus obtain the desired contradiction to the maximum principle and conclude that the condition $\lambda \leq \mu$ is preserved.
\end{proof}

We now turn to the proof of the preservation of convexity. We consider again the function $Z_\sigma$, since it is well known that the condition $|h|^2 \leq \frac{1}{n-1}H^2$, which corresponds to $Z_\sigma \leq 0$ with $\sigma= \frac{1}{n(n-1)}$, implies that all curvatures are nonnegative. We can remark that on an axially stretched hypersurface the two properties are actually equivalent. In fact we have
\begin{eqnarray}
|h|^2 - \frac{1}{n-1}H^2 & = & \lambda^2+(n-1)\mu^2-\frac{1}{n-1}(\lambda+(n-1)\mu)^2 \nonumber \\
& = & \lambda\left( \frac{n-2}{n-1}\lambda  - 2 \mu \right), \label{pinchconvex}
\end{eqnarray}
and the term in parentheses is negative since $\lambda \leq \mu$; thus, the expression is nonpositive if and only if  $\lambda \geq 0$. We now use the maximum principle to show that this inequality is preserved and  becomes strict at all positive times. For later purposes, it is convenient to consider the invariance of the more general inequality $|h|^2 \leq \left(\frac 1n+\sigma \right) H^2$, with $0<\sigma \leq \frac{1}{n(n-1)}$.


\begin{Lem}\label{prespinch}
	Let $M_t$ be a solution of \eqref{Fflow} which is rotationally symmetric and axially stretched. Suppose that $|h|^2 \leq  \left(\frac 1n+\sigma \right) H^2$ at $t=0$ for some $0<\sigma \leq \frac{1}{n(n-1)}$. Then the same inequality holds and becomes strict for all $t>0$.
\end{Lem}
	\begin{proof} We consider again the function $Z_\sigma=|h|^2 - \left(\frac 1n+\sigma \right) H^2$ and analyse the sign of the terms at the right-hand side of \eqref{omogeq} at a point where $G=Z_\sigma$ attains a zero maximum point for the first time. We first observe that, since $|h|^2 = \frac{1}{n}H^2$ at an umbilical point, and $H>0$ everywhere on our hypersurfaces, we have $0 \leq \lambda < \mu$ at our point; in particular, the point is different from the poles. 

As in the previous Lemma, we can use \eqref{reactZ} and \eqref{est-reaction} to conclude that the reaction terms in equation \eqref{omogeq} are nonpositive. The gradient terms, on the other hand, cannot be estimated by the argument of \cite{AMC} because $\sigma$ cannot be chosen arbitrarily small this time. We instead apply Lemma \ref{LemmaLWW} with $F$ equal to our speed and $G=Z_\sigma$, so that $a=\alpha$ and $b=2$. The condition $\lambda \neq \mu$ is satisfied at our point, as we have observed before. In addition, we have
		$$
		\dot{g}^1=\frac{\partial}{\partial \lambda}\left(|h|^2-\frac{H^2}{n-1}\right)=2\left(\lambda-\frac{H}{n-1} \right)<0,$$
		since $\lambda < \mu$. Thus, the hypothesis $\dot{g}^{1} \neq 0$ also holds and we can apply the Lemma. Keeping into account that $G=0$ at $p$, many terms in \eqref{evofg} vanish and we are left with
						\begin{eqnarray*}
						&& [\dot{G}^{ij}\ddot{F}^{lm,rs}-\dot{F}^{ij}\ddot{G}^{ml,rs} ] \nabla_i h_{ml}\nabla_j h_{rs} \\
& = &			\left[2\left( \lambda-\frac{H}{n-1} \right) \frac{\alpha(\alpha-1)F}{\mu^2}-4 \left( \lambda-\frac{H}{n-1} \right)\frac{\alpha F}{\mu(\lambda-\mu)} \right](\nabla_1 h_{22})^2,
		\end{eqnarray*}
		and the quantity in brackets is negative if $\lambda<\mu$.
		
By the maximum principle, we conclude that the inequality $Z_\sigma \leq 0$ is preserved. By the strong maximum principle, the inequality becomes strict for all positive times unless we have $Z_\sigma$ vanishes identically for all times on our solution. Then the right-hand side of \eqref{omogeq} must also vanish, and both the reaction and the gradient terms are identically zero. In particular, \eqref{est-reaction} shows that the only value of $\sigma \in \left(0, \frac{1}{n(n-1)}\right]$ for which this can occur is $\sigma=\frac{1}{n(n-1)}$. But then $Z_\sigma \equiv 0$ implies that $\lambda$ is identically zero by \eqref{pinchconvex} and that $M_t$ is a cylinder, which is a contradiction since $M_t$ is compact. We conclude that $Z_\sigma$ is strictly positive on $M_t$ for $t>0$. 
		\end{proof}	

The previous results show that our solution remains convex and axially stretched until it exists. We will now show that pinching actually improves when the curvature becomes large. Our next statement is the same as in Theorem 11.1 in \cite{AMC}, although we need an independent proof since our assumptions are different. The pinching assumption on the initial value in the next lemma is stronger than what we have assumed until now: however, by the previous lemma, we can assume that it is satisfied after possibly replacing $M_0$ with $M_{t_0}$ for any small $t_0>0$.

\begin{Lem}
	Let $M_0$ be a hypersurface such that $|h|^2-\frac 1n H^2<\sigma_0 H^2$, with $\sigma_0 < \frac{1}{n(n-1)}$. Let $\mathcal{M}_H=\max_{M_0} H$. Then there exists $l \in (0,1)$ such that $|h|^2-\frac 1n H^2\leq \min\left\{\sigma_0 H^2, \sigma_0\mathcal{M}_H^l H^{2-l}\right\}$ along the flow.
\end{Lem}
\begin{proof}
We know from Lemma \ref{prespinch} that the condition $|h|^2\leq \sigma_0 H^2$ is preserved; therefore, all we need to show is that the function
$$
Z_l= |h|^2-\frac 1n H^2 - \sigma_0 \mathcal{M}_H^l H^{2-l}
$$
remains negative at least at the points where $H\geq \mathcal{M}_H$.
To show this, we consider the evolution equation \eqref{omogeq} with $G=Z_l$ and we estimate the right hand side at a zero maximum point for $Z_l$. We omit the analysis of the reaction terms, since it can be done exactly in the same way as in Theorem 11.1 of \cite{AMC}, and we focus instead on the estimation of the gradient terms.

This time we cannot apply directly Lemma \ref{LemmaLWW}, because $G$ is the sum of two terms with different degrees of homogeneity. However, we can adapt the strategy of proof from \cite{LWW} to our situation and obtain again a useful expression for the gradient terms. Our starting point will be the following identity, which is a consequence of the rotational symmetry and is independent of the homogeneity of $F,G$: at any stationary point of $G$ with $\dot g^1 \neq 0$ and $\lambda \neq \mu$ we have, see formula (3.11) in \cite{LWW},

\begin{align}
				(&\dot{G}^{ij}\ddot{F}^{lm,rs}-\dot{F}^{ij}\ddot{G}^{ml,rs})\nabla_i h_{ml}\nabla_j h_{rs}=\nonumber \\
				&\left\{(n-1)^2  \left(\frac{\dot g^2}{\dot g^1} \right)^2
				(\dot{g}^1\ddot{f}^{11}-\dot{f}^1\ddot{g}^{11})
				-2(n-1)^2 \frac{\dot g^2}{\dot g^1} 
				(\dot{g}^1\ddot{f}^{12}-\dot{f}^1\ddot{g}^{12}) \right. \nonumber\\
				&+(n-1) (\dot{g}^1\ddot{f}^{22}-\dot{f}^1\ddot{g}^{22})
				+(n-1) (n-2) (\dot{g}^1\ddot{f}^{23}-\dot{f}^1\ddot{g}^{23})  \nonumber\\
				& \left. +2(n-1)\frac{\dot{g}^2\dot{f}^1-\dot{f}^2\dot{g}^1}{\lambda-\mu}\right\} (\nabla_1h_{22})^2.
 \label{evofg1}
			\end{align}

The property $\lambda \neq \mu$ certainly holds at a point where $G=0$, since $|h|^2-\frac 1n H^2$ at an umbilical point, and $H>0$ everywhere on our hypersurface. We will check later in the proof that the requirement $\dot g^1 \neq 0$ is also satisfied.

Since $F$ is homogeneous of degree $\alpha$, we deduce from Euler's theorem and the rotational symmetry that the following identities hold:
	\begin{align}
		\lambda \dot{f}^1+(n-1)\mu \dot{f}^2 &=\alpha F \label{eqhom1f}\\
		\lambda^2 \ddot{f}^{11}+2(n-1)\lambda\mu\ddot{f}^{12} & \nonumber \\
		+(n-1)\mu^2\ddot{f}^{22}+(n-1)(n-2)\mu^2 \ddot{f}^{23}& =\alpha(\alpha-1)F. \label{eqhom2f}
	\end{align}
When we consider the analogous expressions for $G$, some additional terms occur due to the lack of homogeneity. We can see this either by applying Euler's theorem to the components of $G$ separately, or by a direct computation. To simplify notation, we set $\hat G:= \sigma_0\mathcal{M}_H H^{2-l}$. Then the first derivatives of $G$ have the form
\begin{equation}\label{derprg}
\dot{g}^1=2\left(\lambda-\frac 1n H \right) -\frac{2-l}{H} \hat G, \qquad
	\dot{g}^2=2\left(\mu -\frac 1n H \right)-\frac{2-l}{H} \hat G.
	\end{equation}
We observe that, since $\lambda \leq \mu$ and $\hat G>0$, we have $\dot{g}^1 < 0$ everywhere, so that in particular the condition $\dot{g}^1 \neq 0$ holds at our stationary point, as claimed.	
	
The second derivatives of $g$ are
\begin{align}
\ddot{g}^{ii}&=2-\frac 2n-\frac{(2-l)(1-l)}{H^2} \hat G, \qquad 1 \leq i \leq n\nonumber \\
\ddot{g}^{ij}&=-\frac 2n -\frac{(2-l)(1-l)}{H^2} \hat G, \qquad 1 \leq i \neq j \leq n. \label{dersecg}
	\end{align}
Therefore we have
	\begin{align}
		\lambda \dot{g}^1+(n-1)\mu \dot{g}^2 &=2 G+l \hat{G} \label{eqhom1}\\
\lambda^2 \ddot{g}^{11}+2(n-1)\lambda\mu\ddot{g}^{12} & \nonumber \\
+(n-1)\mu^2\ddot{g}^{22}+(n-1)(n-2)\mu^2 \ddot{g}^{23}& =2G-(l^2-3l)\hat{G}.\label{eqhom2}
\end{align}

From \eqref{eqhom1} we deduce that, at a point where $G=0$,
\begin{equation*}
	(n-1)\frac{\dot{g}^2}{\dot{g}^1}=-\frac{\lambda}{\mu}-\frac{l\hat G}{\mu\dot{g}^1}
\end{equation*}
Using this, we can rewrite the term in brackets in \eqref{evofg1} as follows:
\begin{align*}
&\left\{(\dot{g}^1\ddot{f}^{11}-\dot{f}^1\ddot{g}^{11})\frac{\lambda^2}{\mu^2}+2(n-1)(\dot{g}^1\ddot{f}^{12}-\dot{f}^1\ddot{g}^{12})\frac{\lambda}{\mu}\right.\\
&+(n-1)(\dot{g}^1\ddot{f}^{22}-\dot{f}^1\ddot{g}^{22})+(n-1)(n-2)(\dot{g}^1\ddot{f}^{23}-\dot{f}^1\ddot{g}^{23})\\
&-\left.\frac{2\dot{f}^1\dot{g}^1\lambda}{(\lambda-\mu)\mu} - \frac{2(n-1)\dot{f}^2\dot{g}^1}{\lambda-\mu}\right.\\
&+\left.(\dot{g}^1\ddot{f}^{11}-\dot{f}^1\ddot{g}^{11})\left[\left(\frac{\hat{l G}}{\mu\dot{g}^1}\right)^2+\frac{2\lambda}{\mu}\frac{l\hat{G}}{\mu\dot{g}^1}\right]\right.\\
&\left.+2(n-1)(\dot{g}^1\ddot{f}^{12}-\dot{f}^1\ddot{g}^{12})\frac{l \hat{G}}{\mu \dot{g}^1} -2\frac{\dot{f}^1\dot{g}^1}{(\lambda-\mu)\mu}l \hat{G}\right\}.
\end{align*}
After rearranging the terms with the second derivatives and using \eqref{eqhom1f}, \eqref{eqhom2f}, \eqref{eqhom2} we conclude that, at a zero maximum point for $G$,
	\begin{align}
		(&\dot{G}^{ij}\ddot{F}^{lm,rs}-\dot{F}^{ij}\ddot{G}^{ml,rs})\nabla_i h_{ml}\nabla_j h_{rs}=\nonumber\\
		&\left\{ \dot{g}^{1}\frac{\alpha(\alpha-1)F}{\mu^2}+\dot{f}^{1} \frac{(l^2-3l)\hat{G}}{\mu^2}-2 \dot{g}^{1}\frac{\alpha F}{\mu(\lambda - \mu)}+2\dot{f}^1 \frac{l\hat{G}}{\mu(\lambda - \mu)}\right. \nonumber\\
		& \,\, +  \left(\frac{l^2\hat G^2}{\mu^2(\dot{g}^{1})^2}-2 \frac{l\hat{G}\lambda}{\mu^2\dot{g}^{1}}\right)(\dot{g}^{1}\ddot{f}^{11}-\dot{f}^{1}\ddot{g}^{11}) \nonumber\\
		& \,\, \left. -2(n-1)\frac{l\hat{G}}{\mu \dot{g}^{1}}(\dot{g}^{1}\ddot{f}^{12}-\dot{f}^{1}\ddot{g}^{12})\right\}(\nabla_1 h_{22})^2. \label{gradterms}
	\end{align}
	As observed above, $\dot{g}^1$ is negative. On the other hand, $\dot{f}^1>0$ by assumption. Thus, all terms in the first row inside the brackets are negative.
	
Before considering the remaining terms, let us observe that the derivative $\ddot{g}^{11}$ is positive; in fact,  $G=0$ implies
$$\hat G=|h|^2-\frac 1n H^2<\sigma_0 H^2 < \frac{1}{n(n-1)}H^2$$
and thus  $\hat G H^{-2}< \frac{1}{n(n-1)}$, which implies $\ddot{g}^{11}>0$ by \eqref{dersecg}. Then, we can estimate the terms which involve the second derivatives of $g$ as follows:
	\begin{align*}
	-\left(\frac{l^2\hat{G}^2}{\mu^2(\dot{g}^{1})^2} \, -\right. & \left. 2\frac{l\hat{G}\lambda}{\mu^2\dot{g}^{1}}\right)\dot{f}^1\ddot{g}^{11}+2(n-1)\frac{l\hat{G}}{\mu \dot{g}^{1}}\dot{f}^1\ddot{g}^{12}\\
	< &\, 	 2\frac{l\hat{G}\lambda}{\mu^2\dot{g}^{1}}\dot{f}^1\ddot{g}^{11}+2(n-1)\frac{l\hat{G}}{\mu \dot{g}^{1}}\dot{f}^1\ddot{g}^{12}\\
	= & \, \frac{2l\hat{G}}{\mu^2\dot{g}^1}\dot{f}^1\left(\ddot{g}^{11}\lambda + (n-1)\ddot{g}^{12}\mu\right)\\
	= & \, \frac{2l\hat{G}}{\mu^2\dot{g}^1}\dot{f}^1\left(2(\lambda-\frac 1n H)-(2-l)(1-l)H^{-1}\hat G \right)\\
	= & \, \frac{2l\hat{G}}{\mu^2\dot{g}^1}\dot{f}^1\left(\dot{g}^1+l(2-l)H^{-1}\hat G \right)
	< \frac{2l\hat{G}}{\mu^2}\dot{f}^1.
	\end{align*}
This term is positive, but it is compensated by the good negative term $\dot{f}^{1} \frac{(l^2-3l)\hat{G}}{\mu^2}$ for $l<1$.

We can rearrange the remaining terms by using Euler's identity for $\dot{f}^1$ to obtain
\begin{align*}
&\left(\frac{l^2\hat{G}^2}{\mu^2(\dot{g}^{1})^2}-2\frac{l\hat{G}\lambda}{\mu^2\dot{g}^{1}}\right)\dot{g}^1\ddot{f}^{11}-2(n-1)\frac{l\hat{G}}{\mu\dot{g}^1}\dot{g}^1\ddot{f}^{12}\\
=&\frac{l^2\hat{G}^2}{\mu^2\dot{g}^1}\ddot{f}^{11}-2\frac{l\hat{G}}{\mu^2}\left(\lambda\ddot{f}^{11}+(n-1)\ddot{f}^{12}\mu\right)\\
=&\frac{l^2\hat{G}^2}{\mu^2\dot{g}^1}\ddot{f}^{11}-2\frac{l\hat{G}}{\mu^2}(\alpha-1)\dot{f}^1\\
= &\frac{l\hat{G}}{\mu^2 \dot{g}^1}\left(l\hat{G}\ddot{f}^{11}-2(\alpha-1)\dot{f}^1\dot{g}^1\right).
\end{align*}
Since $\dot{g}^1<0$, we need to show that the quantity in parentheses is positive.
We plug in the explicit expression for $\dot{g}^1$ to obtain
\begin{align*}
	&l\hat{G}\ddot{f}^{11}-2(\alpha-1)\dot{f}^1\dot{g}^1\\
	= \,&\, l\hat{G}\ddot{f}^{11}-2(\alpha-1)\dot{f}^1\left(2(\lambda-\frac 1n H) -(2-l)H^{-1}\hat G \right)\\
	> \, & \, \hat{G}(l \ddot{f}^{11} +2(\alpha-1)\dot{f}^1(2-l)H^{-1}).
\end{align*}
Since both $\ddot{f}^{11}$ and $\dot{f}^1 H^{-1}$ are comparable to $H^{\alpha-2}$, the sum above is positive with the choice of a suitable small $l$. Therefore the right-hand side in \eqref{gradterms} is nonpositive at a zero maximum for $Z_l$, and this concludes our proof.
\end{proof}

%

Once the previous result is established, we can follow the procedure of \cite{AMC} to conclude our proof, since it applies to our setting without changes needed. We briefly recall the steps of the argument for the reader's convenience. The previous Lemma implies that the principal curvatures become comparable at points where the curvature is large. More precisely, one can easily deduce the following property: for any $\epsilon > 0$ there exists a constant $C(\epsilon)$ such that $\mu \leq (1+\epsilon)\lambda + C(\epsilon)$. By a geometric estimate on convex sets (Theorem 3.1 in \cite{AMC}) one deduces that the ratio between inner and outer radius of $M_t$ approaches one as the singular time is approached. This property allows to employ a maximum principle argument to show that the speed goes to infinity with a uniform rate on all $M_t$ and that the profile of $M_t$ becomes spherical up to rescaling, see Sections 11-12 in \cite{AMC}.

\section{Construction of the ancient ovaloids}

In this section we prove the main theorem of this paper:
	\begin{thm}\label{main}
	Let the speed $f$ satisfy assumptions {\rm (F1)-(F4)} in Section 2. Then there exists a compact, convex, axially symmetric ancient solution $M_t$ of the flow \eqref{Fflow}, defined on $I=(-\infty, 0)$, such that: 
		\begin{itemize}
			\item $M_t$ converges to a round point for $t\to 0$;
			\item $M_t$ is not the shrinking sphere;
			\item the family of rescaled flows $S^{-1}M_{S^{1+\alpha} t}$, with $S \to +\infty$, admits a subsequence converging to the standard shrinking cylinder as $t \to -\infty$.
		\end{itemize}
	\end{thm}
	\begin{proof}
As in \cite{W03,HH,LZ}, we define a sequence of approximants, which solve the flow on a bounded time interval and have increasing eccentricity, and which will subconverge to an ancient solution with the desired properties.
More specifically, for any $l \in \mathbb{N}$, we consider a starting hypersurface $M_0^l$ obtained by smoothly capping the cylinder $[-l,l]\times\mathbb{S}^{n-1}$ with spherical caps of radius one, as described in detail in \cite{HH} or \cite{LZ}. Each $M_0^l$ is convex, rotationally symmetric and satisfies $\lambda \leq \mu$ everywhere, thus by Theorem \ref{roundpoint} the corresponding solution $M_t^l$ converges in finite time to a round point.

We denote the axial length and the spherical radius of $M_t^l$ respectively as
$$a^l(t)=\max\limits_{x\in M_t^l}|x_1|, \qquad b^l(t)=\max\limits_{x\in M_t^l} \left( \sum_{i=2}^{n+1} x_i^2\right)^{\frac 12}.$$
The convergence to a round point implies that $a^l(t)/b^l(t) \to 1$ as the singular time is approached.
Since the speed is homogeneous, the flow admits parabolic rescalings of the form $(x,t) \to (\Lambda x,\Lambda^{1+\alpha}t)$ with $\Lambda>0$. For each $l$, we can choose a suitable scaling factor and add a time translation in order to have that:
\begin{enumerate}
	\item $M^l_t$ is defined on an interval of the form $[-T^l, 0)$; \label{formint}
	\item $\frac{a^l(t)}{b^l(t)}\geq 2$ on $[-T^l,-1]$ and $\frac{a^l(-1)}{b^l(-1)}=2$ for every $l \in \mathbb{N}$.\label{fixratio}
\end{enumerate}
From this point, we will work with the rescaled flows and we will keep the same notation since no confusion should occur. 
As a first step, we prove uniform bounds on the geometry of these flows.

\begin{Lem}
For every compact interval of times $K \subset (-\infty, -1]$, there exist constants $b_K,A_K>0$ such that
	$b_K\leq 2 b^l(t) \leq a^l(t) \leq A_K$ for every $l \in \mathbb{N}$ such that $K \subset [-T^l,-1]$.  In addition, $-T^l \to -\infty$ as $l \to +\infty$.
\end{Lem}
\begin{proof}
We observe that, since the speed of the flow points inwards, $a^l(t)$ and $b^l(t)$ are monotone decreasing. In addition, due to convexity and spherical symmetry, these quantities can be compared with the inner and outer radius: it is easily seen that $M^l_t$ encloses the sphere centred at the origin of radius $b^l(t)/2$ and is enclosed by the sphere of radius $2a^l(t)$ (the constants are not optimal, but are enough for our purposes). If we denote by $R_0(t)$ the radius of a sphere evolving by \eqref{Fflow} shrinking at time $t=0$, then we deduce from the avoidance principle that $M^l_t$ intersects the sphere of radius $R_0(t)$ at all times. This  implies
\begin{equation}\label{spheres}
b^l(t)/2 \leq R_0(t) \leq 2a^l(t), \quad\mbox{ for every }t \in [-T_l,0).
\end{equation}

To bound from above the ratio $a^l(t)/b^l(t)$, 
we use a trick from \cite{HH}. Let us fix any $t_0 < -1$. Using $a^l(t_0) \geq 2 b^l(t_0)$, we see that there is a sphere enclosed by $M^l_{t_0}$, with centre at distance $a^l(t_0)/2$ from the origin and radius $b^l(t_0)/4$. In addition, by the monotonicity of $b^l$ and by \eqref{spheres}, we can estimate
$$
\frac 14 b^l(t_0) \geq \frac 14 b^l(-1) = \frac 18 a^l(-1) \geq \frac{1}{16} R_0(-1).
$$
Thus, by the avoidance principle, if we call $T^*$ the time taken by a sphere of radius $\frac{1}{16} R_0(-1)$ to shrink to a point under the flow, we have that $a^l(t) \geq a^l(t_0)/2$ for $t \in [t_0,t_0+T^*]$. This implies, again by the monotonicity of $b^l$,
\begin{equation}\label{ratioHH}
\frac{a^l(t)}{b^l(t)} \geq \frac{a^l(t_0)}{2 b^l(t)} \geq  \frac{a^l(t_0)}{2b^l(t_0)}, \qquad t \in [t_0,t_0+T^*],
\end{equation}
that is, the ratio $\frac{a^l(t)}{b^l(t)}$ cannot be halved faster than a fixed $l$-independent time. Since $\frac{a^l(-1)}{b^l(-1)}=2$ for all $l$, it follows that the ratio is bounded from above on the compact interval $K$. If we combine this property with the one-sided bounds in \eqref{spheres}, we conclude that $a^l(t),b^l(t)$ are both comparable to $R_0(t)$ for $t \in K$, and therefore satisfy the claimed estimate for suitable constants $b_K,A_K$. 

Finally, since by construction $\frac{a^l(-T_l)}{b^l(-T_l)} = l \to +\infty$, the above argument shows that $-T_l$ cannot be contained in any compact set $K$, and thus diverges to $-\infty$.
\end{proof}

In the next step, we obtain a two-sided bound on the speed which will ensure the compactness of the sequence ${M^l_t}$. Here the proof requires new arguments compared with \cite{HH,LZ}, due to the different properties of the flows considered.

\begin{Lem}
For any $T>1$ there exist two constants $c_T, C_T>0$ and a positive integer $l_T$ such that 
$$c_T \leq \min\limits_{M^l_t} F \leq \max\limits_{M^l_t} F \leq C_T$$
for all $t \in [-T,-1]$ and all $l \geq l_T$.
\end{Lem}
\begin{proof}
	To obtain the bound from above we use Theorem 12 in \cite{AMCZ}, which is based on the well known trick by Tso \cite{T}. To apply this theorem, we need to check the hypothesis that along our flow
	$$\sum_{i=1}^n \frac{\partial f}{\partial \lambda_i}\lambda_i^2\geq C f^2,$$
for a positive constant $C$. It suffices to show that the inequality holds on the cone $\Gamma_0$ which includes the possible values of the curvatures of our hypersurfaces; in addition, by homogeneity, it is enough to consider the points of $\Gamma_0$ which lie on the unit sphere. Since $\Gamma_0$ intersected with the unit sphere is a compact subset of the cone $\Gamma$ where $f$ and its first derivatives are positive, we deduce that $\left( \sum_{i=1}^n \frac{\partial f}{\partial \lambda_i}\lambda_i^2 \right)f^{-2}$ has a positive minimum on this set, and such a minimum is the required constant $C$.	
	
	By the previous lemma, there exist $b_T,A_T>0$ such that $b_T \leq 2 b^l(t)\leq a^l(t)\leq A_T$ for all $l$ such that 
	 $-T_l<-2T$ and $t \in [-2T,-1]$. Then Theorem 12 in \cite{AMCZ} gives the estimate
	$$
		F\leq C'\frac{A_T}{b_T}\left(b_T^{-\alpha}+(t+2T)^{-\frac{\alpha}{1+\alpha}}\right), \qquad t \in (-2T,-1].
	$$
	on $M^l_t$ for every $t \in (-2T,-1]$, where the constant $C'$ only depends on $C,\alpha$. Since $(t+2T)^{-\frac{\alpha}{1+\alpha}}$ is bounded for $t \in [-T,-1]$, the estimate from above on $F$ follows.

To prove the estimate from below, we analyse separately the points which are far from the poles and the ones which are near. We  first consider points on $M^l_t$ where the normal direction forms an angle at least $\frac \pi 4$ with the axis of rotation, and therefore we have $|u_x| \leq 1$ in the parametrisation of Section 2. Since by definition  $u \leq b^l(t)$ everywhere, we see from \eqref{curvatures} that $\mu \geq \frac{1}{\sqrt{2}b^l(t)}  \geq \frac{\sqrt 2}{A_T}$ at such a point. By the monotonicity and homogeneity of $f$, we deduce the bound
$$
F=f(\lambda,\mu,\dots,\mu) \geq f \left(0,\frac {\sqrt 2}{A_T},\dots,\frac {\sqrt 2}{A_T} \right)>0,
$$ 
which gives an $l$-independent bound from below. To estimate the points where the angle is less than $\frac \pi 4$, we use the following general result for flows of convex hypersurfaces, see Theorem 14 in \cite{AMCZ}: if we consider for $z \in S^n$ the support function
$$
s(z,t) = \max_{p \in M^n} \langle \phi(p,t),z \rangle
$$
then we have
$$
	F(p_2,t_2)\geq \frac{s(z,t_1)-s(z,t_2)}{(1+\alpha)(t_2-t_1)}
	$$
	for all $z \in S^n$ and $t_2 > t_1$, where $p_2$ is such that $\nu(p_2,t_2)=z$.

We take now any $t_2 \in [-T,-1]$ and any point $p_2$ where the normal $\nu(p_2,t_2)$ has an angle less than $\frac \pi 4$ with the axis of rotation, and we set $z=\nu(p_2,t_2)$. We have that $M^l_{t_2}$ is enclosed in a sphere of radius $2a^l(t_2)$ and therefore 
$s(z,t_2) \leq 2a^l(t_2) \leq 2A_T$. On the other hand, by our assumption on $z$ and by \eqref{spheres}, we have
$$s(z,t_1) \geq \frac{1}{\sqrt 2} a^l(t_1) \geq \frac{R_0(t_1)}{2 \sqrt 2}.$$
We now choose $t_1 <-T$ large enough to have $R_0(t_1) > 8 \sqrt 2 A_T$. For any $l$ such that $-T_l \leq t_1$, we deduce
$$
	F(p_2,t_2)\geq \frac{2 A_T}{(1+\alpha)(t_2-t_1)} > \frac{2 A_T}{(1+\alpha)|t_1|} ,
	$$
which yields a lower bound in this case too, and thus completes the proof of our lemma.
%
%
\end{proof}

\noindent{\em Proof of theorem \ref{main} (conclusion)} 
From the bounds on the speed given in the previous lemma, we immediately obtain bounds on the curvature. In fact, by monotonicity and by the property $\lambda \leq \mu$, we have
$$
f(0,1,\dots,1) \mu^\alpha \leq F \leq  f(1,\dots,1) \mu^\alpha
$$
and therefore the speed and the radial curvature control each other. The axial curvature $\lambda$ on the other hand is bounded above by $\mu$ and below by zero. It follows that the curvatures of all approximating flows are contained in a compact subset of the cone $\Gamma_0$, thus we deduce curvature derivative bounds of any order by the method recalled at the beginning of the proof of Theorem \ref{roundpoint}. In particular, we have compactness of the immersions by Ascoli-Arzel\`a theorem and we can find a subsequence which converges to a compact convex solution $\varphi_{\infty}$ of the flow \ref{Fflow}. The radial curvature $\mu$ is positive everywhere on the limit because of the uniform bounds on the approximants; the axial curvature $\lambda$ is also positive by Lemma \ref{prespinch}.
Since $T^l \to -\infty$, we see that $\varphi_{\infty}$ is ancient; since $\frac{a^l(-1)}{b^l(-1)}=2$, this ratio is $2$ also in the limit, thus $\varphi_{\infty}$ cannot be a shrinking sphere. We can continue the solution for $t \geq -1$ until it shrinks to a round point in finite time, according to Theorem \ref{roundpoint};  by possibly adding a time translation we can make the singular time to be zero, so that our ancient solution is defined in $(-\infty,0)$.

Finally, the statement about the asymptotic description follows by an argument analogous to the one in \cite{LZ}, which we briefly recall for the  convenience of the reader. For $S>0$, we define the parabolic rescalings $\varphi^S_\infty(\cdot, t)=S^{-1}\varphi_\infty(\cdot, S^{1+\alpha}t)$. For a fixed time $t<0$, we have $S^{1+\alpha}t \to -\infty$, and from this it follows that
the ratio $\frac{a_S(t)}{b_S(t)} \to +\infty$ as $S \to \infty$. On the other hand, using comparison with shrinking cylinders and spheres, one proves that ${b_S(t)}$ remains bounded. The sequence satisfies again $S$-uniform bounds up to second order on compact sets as in the previous lemma, and thus it is precompact and admits a subsequential limit as $S \to +\infty$. Using the property that $a_S(t) \to +\infty$, one can show that the limit contains a line, which is a limit for $S \to \infty$ of geodesics connecting the two poles. By applying Cheeger and Gromoll splitting theorem and using convexity and rotational symmetry, we conclude that the limit is a standard cylinder.
\end{proof}
\medskip

\noindent {\bf Acknowledgements}
The first author has been supported by the grant PRIN2017 CUP E68D19000570006 of MIUR (Italian Ministry of Education and Research) and is a member of the group GNSAGA of INdAM (Istituto Nazionale di Alta Matematica). The second author has been supported by MIUR Excellence Department Project awarded to the Department of Mathematics, University of Rome ``Tor Vergata'', CUP E83C18000100006, and by the grant ``Beyond Borders'' CUP E84I19002220005 of the University of Rome ``Tor Vergata'', and is a member of the group GNAMPA of INdAM. 


\bibliographystyle{abbrv}
  \bibliography{biblio}

\begin{thebibliography}{10}

\bibitem{AW}
S.~J. Altschuler and L.~F. Wu.
\newblock Translating surfaces of the non-parametric mean curvature flow with
  prescribed contact angle.
\newblock {\em Calc. Var. Partial Differential Equations}, 2(1):101--111, 1994.

\bibitem{A1}
B.~Andrews.
\newblock Contraction of convex hypersurfaces in {E}uclidean space.
\newblock {\em Calc. Var. Partial Differential Equations}, 2(2):151--171, 1994.

\bibitem{A5}
B.~Andrews.
\newblock Moving surfaces by non-concave curvature functions.
\newblock {\em Calc. Var. Partial Differential Equations}, 39(3):649--657,
  2010.

\bibitem{AGN}
B.~Andrews, P.~Guan, and L.~Ni.
\newblock Flow by powers of the {G}auss curvature.
\newblock {\em Adv. Math.}, 299:174--201, 2016.

\bibitem{AMC}
B.~Andrews and J.~McCoy.
\newblock Convex hypersurfaces with pinched principal curvatures and flow of
  convex hypersurfaces by high powers of curvature.
\newblock {\em Trans. Amer. Math Soc.}, 364(7):3427--3447, 2012.

\bibitem{AMCZ}
B.~Andrews, J.~McCoy, and Y.~Zheng.
\newblock Contracting convex hypersurfaces by curvature.
\newblock {\em Calc. Var. Partial Differential Equations}, 47(3-4):611--665,
  2013.

\bibitem{An1}
S.~Angenent.
\newblock Shrinking doughnuts.
\newblock {\em Progr. Nonlinear Differential Equations Appl.}, 7:21--38, 1992.

\bibitem{B+}
T.~Bourni, J.~Clutterbuck, X.~H. Nguyen, A.~Stancu, G.~Wei, and V.-M. Wheeler.
\newblock Ancient solutions for flow by powers of the curvature in
  $\mathbb{R}^2$.
\newblock {\em Calc. Var. Partial Differential Equations}, 61(2):1--14, 2022.

\bibitem{BLT4}
T.~Bourni, M.~Langford, and G.~Tinaglia.
\newblock Ancient mean curvature flows out of polytopes.
\newblock {\em To appear in Geom. Topol., preprint arXiv:2006.16338}, 2020.

\bibitem{BLT3}
T.~Bourni, M.~Langford, and G.~Tinaglia.
\newblock Convex ancient solutions of mean curvature flow in slab regions.
\newblock {\em in Differential Geometry in the Large, London Mathematical
  Society Lecture Note Series.}, 463, 2020.

\bibitem{BLT}
T.~Bourni, M.~Langford, and G.~Tinaglia.
\newblock Collapsing ancient solutions of mean curvature flow.
\newblock {\em J. Differential Geom.}, 119(2):187--219, 2021.

\bibitem{BrendleDaskaChoi}
S.~Brendle, K.~Choi, and P.~Daskalopoulos.
\newblock Asymptotic behavior of flows by powers of the {G}aussian curvature.
\newblock {\em Acta Math.}, 219(1):1--16, 2017.

\bibitem{BIS}
P.~Bryan, M.~Ivaki, and J.~Scheuer.
\newblock On the classification of ancient solutions to curvature flows on the
  sphere.
\newblock {\em preprint arXiv:1604.01694}, 2016.

\bibitem{Ger}
C.~Gerhardt.
\newblock {\em Curvature problems, Series in Geometry and Topology, vol. 39}.
\newblock International Press, Somerville, MA, 2006.

\bibitem{H2}
R.~Hamilton.
\newblock The formation of singularities in the {R}icci flow.
\newblock {\em Surveys in {D}ifferential {G}eometry}, 2:7--136, 1993.

\bibitem{HH}
R.~Haslhofer and O.~Hershkovits.
\newblock Ancient solutions of the {M}ean {C}urvature {F}low.
\newblock {\em Comm. Anal. Geom.}, 24(3):593--604, 2016.

\bibitem{HIMW}
D.~Hoffman, T.~Ilmanen, F.~Martin, and B.~White.
\newblock Graphical translators for mean curvature flow.
\newblock {\em Calc. Var. Partial Differential Equations}, 58(117), 2019.

\bibitem{Hu}
G.~Huisken.
\newblock Flow by mean curvature of convex surfaces into spheres.
\newblock {\em J. Differential Geom.}, 20:237--266, 1984.

\bibitem{HS}
G.~Huisken and C.~Sinestrari.
\newblock Convex ancient solutions of the {M}ean {C}urvature {F}low.
\newblock {\em J. Differential Geom.}, 101(2):267--287, 2015.

\bibitem{L1}
M.~Langford.
\newblock A general pinching principle for mean curvature flow and
  applications.
\newblock {\em Calc. Var. Partial Differential Equations}, 56(4), 2017.

\bibitem{LL}
M.~Langford and S.~Lynch.
\newblock Sharp one-sided curvature estimates for fully nonlinear curvature
  flows and applications to ancient solutions.
\newblock {\em J. Reine Angew. Math.}, 765:1--33, 2020.

\bibitem{LWW}
H.~Li, X.~Wang, and J.~Wu.
\newblock Contracting axially symmetric hypersurfaces by powers of the
  $\sigma_k$-curvature.
\newblock {\em J. Geom. Anal.}, 31(3):2656--2702, 2021.

\bibitem{LZ}
P.~Lu and J.~Zhou.
\newblock Ancient solutions for {A}ndrews' hypersurface flow.
\newblock {\em J. Reine Angew. Math.}, 2021(771):85--98, 2021.

\bibitem{McMMV}
J.~A. McCoy.
\newblock More mixed volume preserving curvature flows.
\newblock {\em J. Geom. Anal.}, 27(4):3140--3165, 2017.

\bibitem{MMM}
J.~A. McCoy, F.~Y. Mofarreh, and V.-M. Wheeler.
\newblock Fully nonlinear curvature flow of axially symmetric hypersurfaces.
\newblock {\em Nonlinear Differ. Equ. Appl.}, 22(2):325--343, 2015.

\bibitem{P}
G.~Perelman.
\newblock The entropy formula for the {R}icci flow and its geometric
  applications.
\newblock {\em preprint arXiv:math/0211159}, 2002.

\bibitem{R}
S.~Risa.
\newblock Ancient solutions of curvature flows ({P}h{D} thesis).
\newblock {\em Universit\`a di Roma ``Tor Vergata''}, 2018.

\bibitem{RS}
S.~Risa and C.~Sinestrari.
\newblock Ancient solutions of geometric flows with curvature pinching.
\newblock {\em J. Geom. Anal.}, 29(2):1206--1232, 2019.

\bibitem{RS2}
S.~Risa and C.~Sinestrari.
\newblock Strong spherical rigidity of ancient solutions of expansive curvature
  flows.
\newblock {\em Bull. London Math. Soc.}, 52(1):94--99, 2020.

\bibitem{T}
K.~Tso.
\newblock Deforming a hypersurface by its {G}auss-{K}ronecker curvature.
\newblock {\em Comm. Pure Appl. Math.}, 38(6):867--882, 1985.

\bibitem{WA}
X.-J. Wang.
\newblock Convex solutions to the mean curvature flow.
\newblock {\em Ann. of Math.}, 173:1185--1239, 2011.

\bibitem{W03}
B.~White.
\newblock The nature of singularities in mean curvature flow of mean-convex
  sets.
\newblock {\em J. Amer. Math. Soc.}, 16(1):123--138, 2003.

\end{thebibliography}
\end{document}